\documentclass{amsart}
\usepackage{latexsym,amsmath,amssymb}
\usepackage{amsthm}
\usepackage{pslatex}
\usepackage{cite}
\usepackage{color}
\newtheorem{theorem}{Theorem}[section]
\theoremstyle{plain}
\newtheorem{corollary}[theorem]{Corollary}
\newtheorem{example}[theorem]{Example}
\newtheorem{lemma}[theorem]{Lemma}
\newtheorem{proposition}[theorem]{Proposition}
\numberwithin{equation}{section}

\def\nn{\nonumber}

\def\iimp{\int\limits_0^{\infty}}

\def\Dqa{D^{\,q,\alpha}}
\def\Dq{D^{\,q}}
\def\eqa{e_{q,\alpha}}
\def\Eqa{E_{q,\alpha}}
\def\dqa{d_{q,\alpha}}
\def\Iqa{I^{\,q,\alpha}}
\def\Nqa{N_{q,\alpha}}
\def\Gqa{\Gamma_{q,\alpha}}
\def\Bqa{B_{q,\alpha}}
\def\cqa{c_{q,\alpha}}
\def\sqa{s_{q,\alpha}}
\def\fa1{f_\alpha}
\def\Rqa{R_{q,\alpha}}

\begin{document}

\title{$q$-deformed conformable fractional Natural transform}
\author[O.~Herscovici]{Orli Herscovici}
\address{O.~Herscovici\\Department of Mathematics,
University of Haifa,
3498838  Haifa, Israel}
\email{orli.herscovici@gmail.com}

\author[T.~Mansour]{Toufik Mansour}
\address{T.~Mansour\\Department of Mathematics,
University of Haifa,
3498838  Haifa, Israel}
\email{tmansour@univ.haifa.ac.il}

\begin{abstract}
In this paper, we develop a new deformation and
generalization of the Natural integral transform based
on the conformable fractional $q$-derivative. We obtain
transformation of some deformed functions and apply
the transform for solving linear differential equation with
given initial conditions.
\end{abstract}
\maketitle

\noindent{\sc Keywords:} Laplace transform; Sumudu transform;
$q$-deformation; Jackson $q$-derivative; conformable
fractional $q$-derivative; $q$-Leibniz rule;

\noindent{\sc 2010 MSC:}
05A30; 33D05; 44A10; 44A20; 44A30;

\section{Introduction}
Differential equations appears in many problems of Physics,
Engineering, and other sciences. So we need powerful
mathematical tools to handle them. The integral transforms
are one of the widely used techniques applied for solution
of differential equations. One of disadvantages attributed to
Mathematics was the perfect theoretical models describing
unperfect realistic conditions. The deformed  cases of well
known problems may be very useful to overcome this
disadvantage. By varying those deformed parameters, one can
obtain solutions that will be possible from industrial point
of view.

In this paper we develop a new deformation of the Natural integral transform. We define new extensions of some special functions and apply our deformed transform to them. As well we obtain different recurrence relations connecting between transform of a function and its derivatives. Among other tools, a new extension of $q,\alpha$-Taylor series is proposed.
We start here with recalling three integral transforms,
namely the Laplace, the Sumudu, and the Natural transform.
Further we develop a deformation of the Natural transform,
that still has not been developed, and show its application
to some deformed differential equations.

The Laplace transform is one of the
most famous from the integral transforms. It is defined as
\begin{align*}
F(s)=\iimp f(t)e^{-st}dt.
\end{align*}
This transform is very useful in solving differential
equations with given initial and boundary conditions. One of
its important features is the transformation from the time
domain to the frequency domain.

In 1993, Watugala \cite{Watugala1993} proposed a new integral
transform, named Sumudu transform, which is defined as
\begin{align*}
S\{f(t)\}=F_s(u)=\iimp\frac{1}{u}e^{-\frac{t}{u}}f(t)dt.
\end{align*}
Among its properties Watugala remarks an easier visualization.
Belgacem \cite{Belgacem2006} denotes the Sumudu
transform as an ideal tool for solving many engineering
problem. The Sumudu transform, unlike the Laplace transform,
is not focusing on transformation into the frequency domain.
One of its main features is preserving units and scale during
transformation \cite{Belgacem2006}.

In 2008, Khan and Khan \cite{Khan2008} defined a new
transform, called $N$-transform and renamed later to Natural
transform, as following
\begin{align*}
R(u,s)=N(f(t))=\iimp f(ut)e^{-st}dt.
\end{align*}
It is easy to see that in the case $u=1$ we obtain the Laplace
transform, and in the case $s=1$ we obtain the Sumudu
transform. The Natural transform was studied and applied to
Maxwell's equations in the series of works of Belgacem and
Silambarasan \cite{Belgacem2012,Belgacem2012a,
Belgacem2012b}.

For both Sumudu and Laplace transforms their q-analogues were
obtained and studied. The $q$-analogues of the Sumudu
transform based on the Jackson's $q$-derivative and
$q$-integral were studied in \cite{Albayrak2013,Albayrak2013a,
Ucar2014}. The $q$-analogues of the Laplace transform based
some on the Jackson's  and some on the Tsallis's
$q$-derivative and $q$-integral were studied in
\cite{Hahn1949,Bohner2010,Chung2014,Lenzi1999,Plastino2013,
Purohit2007,Ucar2012,Yadav2009}.
But the deformations of the Natural transform were not defined
and not studied. Due to its dual nature and close relationship
with both, Laplace and Sumudu, transforms, the Natural
transform is more flexible and lets easily to choose during
the problem solution, what way is preferable in each concrete
case.

In this paper, we define and study a deformation of the
Natural transform based on the conformable fractional
$q$-derivative defined by Chung \cite{Chung2016}. This
deformation is actually a generalization of the
$q$-deformation based on the Jackson $q$-derivative, and,
thus, our deformation generalizes the $q$-deformed Laplace
transform based on the Jackson $q$-derivative
\cite{Chung2014} and proposes
another definition for the $q$-Sumudu transform different
from \cite{Albayrak2013a}.
We obtain some applications of this deformation of the
Natural transform.

\section{Definitions and some properties of the conformable
$q$-derivative}
In \cite{Chung2016}, Chung defined a conformable fractional
$q$-derivative as
\begin{align}
\Dqa_xf(x)=
\frac{[\alpha](f(x)-f(qx)}{x^\alpha(1-q^\alpha)}
=x^{1-\alpha}D^{\,q}_xf(x),\label{Definition-Dqa}
\end{align}
where $[\alpha]=\dfrac{1-q^\alpha}{1-q}$ is the $q$-number
of $\alpha$, and $D^{\,q}_x$ is the Jackson $q$-derivative
with respect to the variable $x$. This operator is a linear
operator \cite{Chung2016}. It is easy to see that in case
$\alpha=1$, this differential operator coincides with Jackson
$q$-derivative. The following notation is widely used in
$q$-calculus $(a+b)^n_q=\prod_{j=0}^{n-1}(a+q^jb)$.
Accordingly with this definition we have
\begin{align}
(a+b)^n_{q^\alpha}=\left\{\begin{array}{ll}
\prod_{j=0}^{n-1}(a+q^{\alpha j}b),\quad &\text{ for integer }
n>0,\\
1,\quad &\text{ }n=0.\end{array}\right.
\label{npos}
\end{align}

The conformable fractional $q$-integral is an inverse
operation of the conformable fractional $q$-derivative
\begin{align}
I^{\,q,\alpha}_xf(x)=\frac{1}{[\alpha]}
(1-q^\alpha)x^\alpha\sum_{j\geq 0}q^{\alpha j}f(q^jx)
=I^{\,q}_x(x^{\alpha-1}f(x)),\nn
\end{align}
where $I^{\,q}_x$ is the Jackson $q$-integral.
Then for $\alpha$-monomial $x^{\alpha n}$ we have
\begin{align}
\Dqa_x x^{\alpha n}=[n\alpha]x^{\alpha(n-1)},\hspace{1cm}
\Iqa_x x^{\alpha n}=\frac{x^{\alpha(n+1)}}{[(n+1)\alpha]}.
\label{dqa-monomial}
\end{align}
The Leibniz rule for the conformable fractional $q$-derivative
has the following form (see \cite{Chung2016})
\begin{align}
\Dqa_x(f(x)g(x))=f(qx)\Dqa_xg(x)+\left(\Dqa_xf(x)\right)g(x).
\label{qaLeibnizRule}
\end{align}

Therefore, by integrating both sides of \eqref{qaLeibnizRule},
we obtain a rule for integrating by parts
\begin{align}
\int\left(\Dqa_x f(x)\right)g(x)\dqa x=f(x)g(x)
-\int f(qx)\Dqa_xg(x)\dqa x. \label{qaIntegrationByParts}
\end{align}
Chung defined also a conformable fractional $q$-exponential
function
\begin{align*}
\eqa(x)=\sum_{j\geq 0}\frac{x^{\alpha j}}{[j\alpha]!}
=((1-q)x^\alpha;q^\alpha)_\infty,
\end{align*}
where $[n\alpha]!=[\alpha]\cdot[2\alpha]\cdots
[n\alpha]$, with the property
\begin{align}
\Dqa_x\eqa(ax)=a^\alpha\eqa(ax).\label{eqa-Derivative}
\end{align}
Note that, as usual, $[0]!=1$.

Two new deformations of trigonometric functions were proposed
in \cite{Chung2016}
\begin{align}
\eqa(i^{\frac{1}{\alpha}}x)=\cqa(x)+i\sqa(x),
\label{qa-e-sin-cos}
\end{align}
where
\begin{align*}
\cqa(x)=\sum_{n\geq 0}\frac{(-1)^n}{[2n\alpha]!}x^{2\alpha n},
\quad\sqa(x)=\sum_{n\geq 0}\frac{(-1)^n}{[(2n+1)\alpha]!}
x^{\alpha(2n+1)}.
\end{align*}
From \eqref{qa-e-sin-cos} one can obtain
\begin{align}
\cqa(x)&=\frac{1}{2}\left(\eqa(i^\frac{1}{\alpha}x)+
\eqa((-i)^\frac{1}{\alpha}x)\right)\label{cqa-exp},\\
\sqa(x)&=\frac{1}{2i}\left(\eqa(i^\frac{1}{\alpha}x)-
\eqa((-i)^\frac{1}{\alpha}x)\right).\label{sqa-exp}
\end{align}
By applying \eqref{eqa-Derivative},
it is easy to show that	
\begin{align*}
\Dqa_x\cqa(x)=-\sqa(x),\qquad \Dqa_x\sqa(x)=\cqa(x).
\end{align*}
By using the definition of the deformed conformable fractional
derivative \eqref{Definition-Dqa}, we evaluate conformable
derivative of a function $\dfrac{1}{\eqa(ax)}$.
\begin{align*}
\Dqa_x\frac{1}{\eqa(ax)}&=x^{1-\alpha}\Dq_x\frac{1}{\eqa(ax)}
=x^{1-\alpha}\frac{\frac{1}{\eqa(ax)}-\frac{1}{\eqa(qax)}}
{x-qx}\nn\\
&=x^{1-\alpha}
\frac{\eqa(qax)-\eqa(ax)}{\eqa(qax)\eqa(ax)(x-qx)}
=-\frac{\Dqa_x\eqa(ax)}{\eqa(qax)\cdot\eqa(ax)},
\end{align*}
and, by applying \eqref{eqa-Derivative}, we obtain
\begin{align}
\Dqa_x\frac{1}{\eqa(ax)}
&=-\frac{a^\alpha\eqa(ax)}{\eqa(qax)\cdot\eqa(ax)}
=-\frac{a^\alpha}{\eqa(qax)}.\label{inverse-eqa-Derivative}
\end{align}

It is easy to see that $\Dqa_xC=0$, where $C$ is a constant
($C$ does not depend on $x$). Indeed, from
\eqref{Definition-Dqa}, we have
$\Dqa_xC=x^{1-\alpha}\Dq_xC=0$.

We can see that $\Dqa_xx^\alpha=[\alpha]$. We would like to
build a sequence of polynomials $P_0(x),P_1(x),\ldots,P_n(x)$
of degrees $0,\alpha, \ldots, n\alpha$ respectively, so that
\begin{align*}
&\Dqa_xP_n(x)=P_{n-1}(x),\\
&P_n(a^\frac{1}{\alpha})=0,
\end{align*}
with initial condition $P_0(x)=1$. Therefore, the polynomial
$P_1(x)$ has the following form $P_1(x)=(x^\alpha-a)/[\alpha]$
. Obviously, $P_1(a^\frac{1}{\alpha})=((a^\frac{1}{\alpha})^
\alpha-a)/[\alpha]=0$ and $\Dqa_xP_1(x)=1=
P_0(x)$.

\begin{proposition}\label{prop1}
For all natural $n$ the following relation holds
\begin{align*}
\Dqa_x (x^\alpha-a)^n_{q^\alpha}=[n\alpha]
(x^\alpha-a)^{n-1}_{q^\alpha}.
\end{align*}
\end{proposition}
\begin{proof}
We proceed the proof by induction on $n$.
It is easily to see that for $n=1$ we have
\begin{align*}
\Dqa_x(x^\alpha-a)^1_{q^\alpha}=\Dqa_x(x^\alpha-a)=[\alpha],
\end{align*}
and our statement holds.
Let us assume that our statement holds for some integer $k$.
We will prove it for $k+1$.
By \eqref{npos}, we have
$(x^\alpha-a)^{k+1}_{q^\alpha}=(x^\alpha-a)^k_{q^\alpha}
(x^\alpha-q^{\alpha k}a)$. By applying \eqref{qaLeibnizRule},
we obtain
\begin{align*}
\Dqa_x(x^\alpha-a)^{k+1}_{q^\alpha}&=
\Dqa_x\left((x^\alpha-a)^k_{q^\alpha}
(x^\alpha-q^{\alpha k}a)\right)\nn\\
&=(q^\alpha x^\alpha-q^{\alpha k}a)\cdot[k\alpha]
(x^\alpha-a)^{k-1}_{q^\alpha}+[\alpha](x^\alpha-a)^k_{q^\alpha}
\nn\\
&=q^\alpha(x^\alpha-q^{\alpha(k-1)}a)\cdot[k\alpha]
(x^\alpha-a)^{k-1}_{q^\alpha}+[\alpha](x^\alpha-a)^k_{q^\alpha}
\nn\\
&=q^\alpha[k\alpha](x^\alpha-a)^k_{q^\alpha}
+[\alpha](x^\alpha-a)^k_{q^\alpha}
\nn\\
&=(x^\alpha-a)^k_{q^\alpha}\left(q^\alpha\cdot
\frac{1-q^{k\alpha}}{1-q}+\frac{1-q^\alpha}{1-q}\right)\nn\\
&=(x^\alpha-a)^k_{q^\alpha}\cdot
\frac{q^\alpha-q^{\alpha(k+1)}+1-q^\alpha}{1-q}\nn\\
&=(x^\alpha-a)^k_{q^\alpha}\cdot
\frac{1-q^{\alpha(k+1)}}{1-q}=[(k+1)\alpha]
(x^\alpha-a)^k_{q^\alpha},
\end{align*}
which completes the induction.
\end{proof}
Let us consider now $(a-x^\alpha)^n_{q^\alpha}$.
By \eqref{npos}, we have
\begin{align*}
(a-x^\alpha)^n_{q^\alpha}&=(a-x^\alpha)(a-q^\alpha x^\alpha)
(a-q^{2\alpha}x^\alpha)\cdot\ldots\cdot
(a-q^{(n-1)\alpha}x^\alpha)\nn\\
&=(a-x^\alpha)q^\alpha(q^{-\alpha}a- x^\alpha)
q^{2\alpha}(q^{-2\alpha}a-x^\alpha)\cdot\ldots\cdot
q^{(n-1)\alpha}(q^{-(n-1)\alpha}a-x^\alpha)\nn\\
&=(-1)^nq^\frac{\alpha n(n-1)}{2}(x^\alpha-q^{-(n-1)\alpha}a)
\cdot\ldots(x^\alpha-q^{-2\alpha}a)(x^\alpha-q^{-\alpha}a)
(x^\alpha-a)\nn\\
&=(-1)^nq^\frac{\alpha n(n-1)}{2}
(x^\alpha-q^{-(n-1)\alpha}a)^n_{q^\alpha}.
\end{align*}
Now, by using Proposition~\ref{prop1}, we obtain
\begin{align}
\Dqa_x (a-x^\alpha)^n_{q^\alpha}&=\Dqa_x
\left((-1)^nq^\frac{\alpha n(n-1)}{2}
(x^\alpha-q^{-(n-1)\alpha}a)^n_{q^\alpha}\right)\nn\\
&=[n\alpha](-1)^nq^\frac{\alpha n(n-1)}{2}
(x^\alpha-q^{-(n-1)\alpha}a)^{n-1}_{q^\alpha}\nn\\
&=(-1)^n[n\alpha]q^\alpha q^{2\alpha}\ldots q^{(n-1)\alpha}
(x^\alpha-q^{-(n-1)\alpha}a)
\cdot\ldots(x^\alpha-q^{-2\alpha}a)(x^\alpha-q^{-\alpha}a)
\nn\\
&=-[n\alpha]q^\alpha q^{2\alpha}\ldots q^{(n-1)\alpha}
(q^{-(n-1)\alpha}a-x^\alpha)
\cdot\ldots(q^{-2\alpha}a-x^\alpha)(q^{-\alpha}a-x^\alpha)
\nn\\
&=-[n\alpha]
(a-q^{(n-1)\alpha}x^\alpha)
\cdot\ldots(a-q^{2\alpha}x^\alpha)(a-q^\alpha x^\alpha)
\nn\\
&=-[n\alpha](a-q^\alpha x^\alpha)^{n-1}_{q^\alpha}.
\nn
\end{align}
Thus, we can state the following result.
\begin{proposition} \label{Prop-DiffRule2}
For all $n\geq1$,
\begin{align*}
\Dqa_x(a-x^\alpha)^n_{q^\alpha}
=-[n\alpha](a-q^\alpha x^\alpha)^{n-1}_{q^\alpha}.
\end{align*}
\end{proposition}
Now, we can state that $P_n(x)=\dfrac{(x^\alpha-a)^n_
{q^\alpha}}{[n\alpha]!}$. Indeed, $P_n(a^\frac{1}{\alpha})=0$,
and
\begin{align*}
\Dqa_xP_n(x)=\Dqa_x\frac{(x^\alpha-a)^n_{q^\alpha}}
{[n\alpha]!}=\frac{[n\alpha](x^\alpha-a)^{n-1}_{q^\alpha}}
{[n\alpha]!}=\frac{(x^\alpha-a)^{n-1}_{q^\alpha}}
{[(n-1)\alpha]!}=P_{n-1}(x).
\end{align*}
Therefore, by using the
\cite[Theorem 2.1 and Theorem~8.1]{Kac2002},
we can state the following result.
\begin{theorem}
Any polynomial or formal power series function $f(x)$ can be
expressed via the generalized conformable fractional
$q$-Taylor expansion about
$x=a^\frac{1}{\alpha}$ as
\begin{align*}
f(x)=\sum\limits_{n\geq 0}(\Dqa_x)^nf(a^\frac{1}{\alpha})\cdot
\frac{(x^\alpha-a)^n_{q^\alpha}}{[n\alpha]!}.
\end{align*}
\end{theorem}

Let us define one more $q$-deformed conformable fractional
exponential function
\begin{align}
\Eqa(x)=\sum_{j\geq 0}q^{\alpha j(j-1)/2}	
\frac{x^{\alpha j}}{[j\alpha]!}
=(1+(1-q)x^\alpha)^\infty_{q^\alpha}. 
\label{Definition-Eqa}
\end{align}
It is easy to see that $\Eqa(0)=1$. By using
\eqref{dqa-monomial}, we obtain that
\begin{align*}
\Dqa_x\Eqa(ax)&=\sum_{j\geq 1}q^{\alpha j(j-1)/2}	
\frac{a^{\alpha j}x^{\alpha (j-1)}}{[(j-1)\alpha]!}
=\sum_{j\geq 0}q^{\alpha (j+1)j/2}	
\frac{a^{\alpha (j+1)}x^{\alpha j}}{[j\alpha]!}\\
&=a^\alpha\sum_{j\geq 0}q^{\alpha j(j-1)/2}	
\frac{q^{\alpha j}a^{\alpha j}x^{\alpha j}}{[j\alpha]!}
=a^\alpha\Eqa(qax).
\end{align*}
Let us define the $q$-deformed conformable fractional Gamma function for some
$n\geq1$:
\begin{align}
\Gqa(n+1)=\int_{0}^{\infty}x^{\alpha n}\frac{1}{\eqa(qx)}
\dqa x.
\label{Definition-Gamma-q-a}
\end{align}

\begin{proposition}
For all $n>0$, the function $\Gqa(n+1)$ defined by
\eqref{Definition-Gamma-q-a} satisfies the recurrence relation
\begin{align*}
\Gqa(n+1)=[n\alpha]\Gqa(n),
\end{align*}
with the initial condition $\Gqa(1)=1$.
\end{proposition}
\begin{proof}
We proceed the proof by induction on $n$. For $n=0$ we have
\begin{align*}
\Gqa(1)&=\int_{0}^{\infty}\frac{1}{\eqa(qx)}\,\dqa x
=-\int_{0}^{\infty}\Dqa_x\frac{1}{\eqa(x)}
=-\frac{1}{\eqa(x)}\Big|_0^\infty
=1.
\end{align*}
Let us assume that the claim holds for $k-1$, and let us
prove it for $k$.
Let us consider now the function $\Gqa(k+1)$ for some $k$.
By \eqref{Definition-Gamma-q-a}, we have
\begin{align*}
\Gqa(k+1)=\int_{0}^{\infty}x^{\alpha k}\frac{1}{\eqa(qx)}\,
\dqa x,
\end{align*}
from where, by rearranging and using \eqref{qaLeibnizRule},
we get
\begin{align*}
\Gqa(k+1)&=-
\int_{0}^{\infty}x^{\alpha k}
\left(\Dqa_x\frac{1}{\eqa(x)}\right)\,\dqa x\nn\\
&=-x^{\alpha k}\cdot\frac{1}{\eqa(x)}\Big|_0^\infty+
\int_{0}^{\infty}\left(\Dqa_x x^{\alpha k}\right)
\frac{1}{\eqa(qx)}\,\dqa x\nn\\
&=[k\alpha]\int_{0}^{\infty}
x^{\alpha(k-1)}\frac{1}{\eqa(qx)}\,\dqa x\nn\\
&=[k\alpha]\Gqa(k),
\end{align*}
which completes the proof.
\end{proof}
One can immediately obtain from the last proposition the following
result.
\begin{corollary}\label{Prop-Gqa-factorial}
For all natural $n$ it holds that
\begin{align*}
\Gqa(n+1)=[n\alpha]! 
\end{align*}
\end{corollary}

The function $\Gqa(n)$ defined as \eqref{Definition-Gamma-q-a}
is a $q$-deformed conformable fractional extension of the
$\Gamma$-function. It is well known that $\Gamma$-function
is closely related to the $B$-function, that is for the
$B$-function defined as
\begin{align*}
B(m,n)=\int_{0}^{1}x^{m-1}(1-x)^{n-1}dx,
\end{align*}
holds that
\begin{align}
B(m,n)=\frac{\Gamma(m)\Gamma(n)}{\Gamma(m+n)}.
\label{Beta-Gamma-relation}
\end{align}
Let us define the following function
\begin{align*}
\Bqa(m,n)=\int_{0}^{1}x^{\alpha(m-1)}(1-q^\alpha x^\alpha)^
{n-1}_{q^\alpha}\dqa x.
\end{align*}
With the notations $f(qx)=(1-q^\alpha x^\alpha)^{n-1}_
{q^\alpha}$ and $\Dqa_xg(x)=x^{\alpha(m-1)}\dqa x$, we obtain
$f(x)=(1-x^\alpha)^{n-1}_{q^\alpha}$ and
$g(x)=\dfrac{x ^{\alpha m}}{[\alpha m]}$. Therefore, by using
Proposition \ref{Prop-DiffRule2}, we get $\Dqa_xf(x)=-[(n-1)\alpha]
(1-q^\alpha x^\alpha)^{n-2}_{q^\alpha}$. Applying
\eqref{qaIntegrationByParts} with our notations yields
\begin{align}
\Bqa(m,n)&=(1-x^\alpha)^{n-1}_{q^\alpha}\frac{x^{\alpha m}}
{[\alpha m]}\Big|_0^1+\int_{0}^{1}[(n-1)\alpha](1-q^\alpha
x^\alpha)^{n-2}_{q^\alpha}\frac{x^{\alpha m}}{[\alpha m]}\dqa
x\nn\\
&=\frac{[(n-1)\alpha]}{[\alpha m]} \int_{0}^{1}x^{\alpha m}
(1-q^\alpha x^\alpha)^{n-2}_{q^\alpha}\dqa x\label{a2}.
\end{align}
Thus, by assuming $m$ and $n$ are natural numbers, we obtain
\begin{align}
\Bqa(m,n)&=\int_{0}^{1}x^{\alpha(m-1)}(1-q^\alpha x^\alpha)^
{n-1}_{q^\alpha}\dqa x
=\frac{[(n-1)\alpha]}{[\alpha m]} \int_{0}^{1}x^{\alpha m}
(1-q^\alpha x^\alpha)^{n-2}_{q^\alpha}\dqa x\nn\\
&=\frac{[(n-1)\alpha]}{[\alpha m]}\frac{[(n-2)\alpha]}
{[\alpha (m+1)]}\int_{0}^{1}x^{\alpha (m+1)}
(1-q^\alpha x^\alpha)^{n-3}_{q^\alpha}\dqa x\nn\\
&=\ldots\nn\\
&=\frac{[(n-1)\alpha]}{[m\alpha ]}\frac{[(n-2)\alpha]}
{[(m+1)\alpha ]}\ldots\frac{[2\alpha]}{[(m+n-3)\alpha]}
\int_{0}^{1}x^{\alpha (m+n-3)}
(1-q^\alpha x^\alpha)^1_{q^\alpha}\dqa x\nn\\
&=\frac{[(n-1)\alpha]}{[m\alpha ]}\frac{[(n-2)\alpha]}
{[(m+1)\alpha ]}\ldots\frac{[2\alpha]}{[(m+n-3)\alpha]}\nn\\
&\cdot
\left(\frac{x^{\alpha(m+n-2)}}{[(m+n-2)\alpha]}(1-x^\alpha)
\Big|_0^1+\int_0^1[\alpha]\frac{x^{\alpha(m+n-2)}}
{[(m+n-2)\alpha]}\,\dqa x\right)\nn\\
&=\frac{[(n-1)\alpha]\cdot[(n-2)\alpha]\cdot\ldots[2\alpha]
\cdot[\alpha]}{[m\alpha]\cdot[(m+1)\alpha]\cdot\ldots\cdot
[(m+n-3)\alpha]\cdot[(m+n-2)\alpha]}\cdot
\frac{x^{\alpha(m+n-1)}}{[\alpha(m+n-1)]}\Big|_0^1\nn\\
&=\frac{[(n-1)\alpha]!\cdot[(m-1)\alpha]!}{[(m+n-1)\alpha]!}
\nn\\
&=\frac{\Gqa(n)\Gqa(m)}{\Gqa(m+n)}.
\label{qaBeta-Gamma-relation}
\end{align}
It is easy to see that for $\alpha=q=1$,
\eqref{qaBeta-Gamma-relation} turns into
\eqref{Beta-Gamma-relation}. Thus our functions $\Gqa$ and
$\Bqa$ are $q$-deformed conformable fractional extensions of
the well known $\Gamma-$ and $B-$function, respectively.
We can state the following result.
\begin{proposition}
For all natural $m$, $n$ it holds that
\begin{align*}
\Bqa(m,n)=\frac{\Gqa(m)\Gqa(n)}{\Gqa(m+n)}.
\end{align*}
\end{proposition}
This proposition may be extended for all positive $m,n$.

\section{$q$-deformed conformable fractional natural
transform}
We define now a $q$-deformed conformable fractional
natural transform as
\begin{align}
\Nqa(f(t))=\int_0^\infty f(ut)\frac{1}{\eqa(qst)}\,\dqa t,
\hspace{1cm}s>0. \label{Definition-Nqa}
\end{align}
Then we have
\begin{align*}
\Nqa(1)&=\int_0^\infty \frac{1}{\eqa(qst)}\,\dqa t
=-\frac{1}{s^\alpha}\int_0^\infty \Dqa_t\frac{1}{\eqa(st)}\,
\dqa t\nn\\
&=-\frac{1}{s^\alpha}\frac{1}{\eqa(st)}\Big|_0^\infty
=\frac{1}{s^\alpha}.
\end{align*}
Let us now to obtain a transform of $\alpha$-monomial.
\begin{align*}
\Nqa(t^{\alpha N})&=\int_0^\infty u^{\alpha N}t^{\alpha N}
\frac{1}{\eqa(qst)}\,\dqa t\nn\\
&=-\frac{1}{s^\alpha}u^{\alpha N}\int_0^\infty
\left(\Dqa_t\frac{1}{\eqa(st)}\right)t^{\alpha N}\dqa t.
\end{align*}
Integrating by parts of the last equation leads to
\begin{align}
\Nqa(t^{\alpha N})&=-\frac{u^{\alpha N}}{s^\alpha}\left\{
\frac{1}{\eqa(st)}\,t^{\alpha N}\Big|_0^\infty
-\int_0^\infty\frac{1}{\eqa(qst)}\left(\Dqa_t
t^{\alpha N}\right)\dqa t\right\}\nn\\
&=\frac{u^{\alpha N}}{s^\alpha}\int_0^\infty[N\alpha]
\,t^{\alpha(N-1)}\frac{1}{\eqa(qst)}\,\dqa t\nn\\
&=\frac{u^\alpha}{s^\alpha}[N\alpha]\int_0^\infty
u^{\alpha(N-1)}t^{\alpha(N-1)}\frac{1}{\eqa(qst)}\,\dqa t\nn\\
&=\frac{u^\alpha}{s^\alpha}[N\alpha]\Nqa(t^{\alpha(N-1)}
\label{rec1}.
\end{align}
Thus, by \eqref{rec1}, we obtain
\begin{align}
\Nqa(t^{\alpha N})&=\frac{u^\alpha}{s^\alpha}[N\alpha]\cdot
\frac{u^\alpha}{s^\alpha}[(N-1)\alpha]\cdot\ldots
\cdot\frac{u^\alpha}{s^\alpha}[\alpha]\cdot
\frac{1}{s^\alpha}=\frac{u^{\alpha N}}{s^{\alpha(N+1)}}[N\alpha]!,
\label{transf-tan}
\end{align}
or, by applying Corollary~\ref{Prop-Gqa-factorial},
\begin{align*}
\Nqa(t^{\alpha N})=\frac{u^{\alpha N}}{s^{\alpha (N+1)}}
\Gqa(N+1).
\end{align*}
Hence, we can state the following result.
\begin{proposition}
For all integer $N\geq 0$,
\begin{align*}
\Nqa(t^{\alpha N})=
\frac{u^{\alpha N}}{s^{\alpha(N+1)}}\Gqa(N+1).
\end{align*}
\end{proposition}

Now let us consider the transform of the two deformed
exponential functions.
\begin{align}
\Nqa(\eqa(at))&=\int_0^\infty\eqa(aut)\frac{1}{\eqa(qst)}\,
\dqa t\nn\\
&=\int_0^\infty\sum_{n=0}^\infty
\frac{a^{\alpha n}u^{\alpha n}t^{\alpha n}}{[n\alpha]!}
\frac{1}{\eqa(qst)}\,\dqa t\nn\\
&=\sum_{n=0}^\infty\frac{a^{\alpha n}}{[n\alpha]!}
\int_0^\infty u^{\alpha n}t^{\alpha n}\frac{1}{\eqa(qst)}\,
\dqa t\nn\\
&=\sum_{n=0}^\infty\frac{a^{\alpha n}}{[n\alpha]!}
\Nqa(t^{\alpha n})=\sum_{n=0}^\infty\frac{a^{\alpha n}}
{[n\alpha]!}\cdot
\frac{u^{\alpha n}}{s^{\alpha(n+1)}}[n\alpha]!
\nn\\
&=\frac{1}{s^\alpha}\sum_{n=0}^\infty
\frac{(au)^{\alpha n}}{s^{\alpha n}}
=\frac{1}{s^\alpha}\sum_{n=0}^\infty
\left(\frac{au}{s}\right)^{\alpha n}
\nn\\
&=\frac{1}{s^\alpha}
\frac{1}{1-\left(\frac{au}{s}\right)^\alpha}
=\frac{1}{s^\alpha}
\frac{s^\alpha}{s^\alpha-a^\alpha u^\alpha}
=\frac{1}{s^\alpha-a^\alpha u^\alpha},\nn
\end{align}
and
\begin{align}
\Nqa(\Eqa(at))&=\int_0^\infty\Eqa(aut)\frac{1}{\eqa(qst)}\,
\dqa t\nn\\
&=\int_0^\infty\sum_{n=0}^\infty q^\frac{\alpha n(n-1)}{2}
\frac{a^{\alpha n}u^{\alpha n}t^{\alpha n}}{[n\alpha]!}
\frac{1}{\eqa(qst)}\,\dqa t\nn\\
&=\sum_{n=0}^\infty q^\frac{\alpha n(n-1)}{2}
\frac{a^{\alpha n}}{[n\alpha]!}\int_0^\infty
u^{\alpha n}t^{\alpha n}\frac{1}{\eqa(qst)}\,\dqa t\nn\\
&=\sum_{n=0}^\infty q^\frac{\alpha n(n-1)}{2}
\frac{a^{\alpha n}}{[n\alpha]!}
\frac{u^{\alpha n}}{s^{\alpha(n+1)}}
[n\alpha]!\nn\\
&=\sum_{n=0}^\infty q^\frac{\alpha n(n-1)}{2}
\frac{(ua)^{\alpha n}}{s^{\alpha(n+1)}}.\nn
\end{align}
So we can state the following proposition.
\begin{proposition} \label{Transform-of_Exponentials}
The $q$-deformed conformable natural transforms of the
$q$-deformed conformable exponential functions are given by
\begin{align}
\Nqa(\eqa(at))&=\frac{1}{s^\alpha-a^\alpha u^\alpha},
\nn\\
\Nqa(\Eqa(at))&=\sum_{n=0}^\infty q^\frac{\alpha n(n-1)}{2}
\frac{(ua)^{\alpha n}}{s^{\alpha(n+1)}}.\nn
\end{align}
\end{proposition}

Let us consider now the transform of the deformed
trigonometric functions \eqref{cqa-exp},
\eqref{sqa-exp}.
\begin{align}
\Nqa(\cqa(at))&=\frac{1}{2}\left(\Nqa(\eqa(i^\frac{1}{\alpha}
at)+\Nqa(\eqa((-i)^\frac{1}{\alpha}at))\right)\nn\\
&=\frac{1}{2}\left(\frac{1}{s^\alpha-
(i^\frac{1}{\alpha}a)^\alpha u^\alpha}+\frac{1}
{s^\alpha -((-i)^\frac{1}{\alpha}a)^\alpha u^\alpha}\right)
\nn\\
&=\frac{1}{2}\left(\frac{1}{s^\alpha-i(au)^\alpha}
+\frac{1}{s^\alpha +i(au)^\alpha}\right)\nn\\
&=\frac{1}{2}\frac{s^\alpha +i(au)^\alpha+s^\alpha-i
(au)^\alpha}{(s^\alpha-i(au)^\alpha)(s^\alpha
+i(au)^\alpha)}\nn\\
&=\frac{s^\alpha}{s^{2\alpha}+(au)^{2\alpha}}.\nn
\end{align}
In the same way we obtain
\begin{align}
\Nqa(\sqa(at))&=\frac{1}{2i}\left(\Nqa(\eqa(i^\frac{1}{\alpha}
at)-\Nqa(\eqa((-i)^\frac{1}{\alpha}at))\right)\nn\\
&=\frac{1}{2i}\left(\frac{1}{s^\alpha-
(i^\frac{1}{\alpha}a)^\alpha u^\alpha}-\frac{1}
{(s^\alpha -((-i)^\frac{1}{\alpha}a)^\alpha u^\alpha}\right)
\nn\\
&=\frac{1}{2i}\left(\frac{1}{s^\alpha-i(au)^\alpha}
-\frac{1}{s^\alpha +i(au)^\alpha}\right)\nn\\
&=\frac{1}{2i}\frac{s^\alpha +i(au)^\alpha-s^\alpha+i
(au)^\alpha}{(s^\alpha-i(au)^\alpha)(s^\alpha
+i(au)^\alpha)}\nn\\
&=\frac{(au)^\alpha}{s^{2\alpha}+(au)^{2\alpha}},\nn
\end{align}
and we can state the following proposition.
\begin{proposition} \label{Transform-of-Trigo}
The deformed conformable fractional natural transform
of deformed trigonomeric functions defined by \eqref{cqa-exp}
and \eqref{sqa-exp} is given by
\begin{align}
\Nqa(\cqa(t))&=\frac{s^\alpha}{s^{2\alpha}+u^{2\alpha}},
\nn\\
\Nqa(\sqa(t))&=\frac{u^\alpha}{s^{2\alpha}+u^{2\alpha}}.
\nn
\end{align}
\end{proposition}

Suppose that function $f(t)$ has a polynomial or formal power
series expansion in $\alpha$-monomials $t^{\alpha n}$. Let us
denote such function $f(t)$ by $\fa1(t)$. We
consider now the transform of a derivative
\begin{align}
\Nqa(\Dqa_t\fa1(t))&=\int_0^\infty\left(\Dqa_t\fa1\right)(ut)
\frac{1}{\eqa(qst)}\,\dqa t\nn\\
&=\frac{1}{u^\alpha}\int_0^\infty\left(\Dqa_t\fa1\right)(y)
\frac{1}{\eqa(qs\frac{y}{u})}\,\dqa y\nn\\
&=\frac{1}{u^\alpha}\fa1(y)\frac{1}{\eqa(qs\frac{y}{u})}
\Big|_0^\infty
+\frac{1}{u^\alpha}\frac{s^\alpha}{u^\alpha}\int_0^\infty \fa1(y)
\frac{1}{\eqa(qs\frac{y}{u})}\,\dqa y\nn\\
&=-\frac{1}{u^\alpha}\fa1(0)+\frac{s^\alpha}{u^\alpha}\int_0^\infty
\fa1(ut)\frac{1}{\eqa(qst)}\,\dqa t
\nn\\
&=-\frac{1}{u^\alpha}\fa1(0)+\frac{s^\alpha}{u^{\alpha}}
\Nqa(\fa1(t)).\nn
\end{align}
Let us rewrite it as
\begin{align}
\Nqa(\Dqa_t \fa1(t))=\frac{s^\alpha}{u^\alpha}\Nqa
(\fa1(t))-\frac{1}{u^\alpha}\fa1(0).\label{TransformOfDerivative}
\end{align}
Therefore,
\begin{align}
\Nqa((\Dqa_t)^2\fa1((t))&=\frac{s^\alpha}
{u^\alpha}\Nqa(\Dqa_t\fa1(t))-\frac{1}{u^\alpha}(\Dqa_t\fa1)(0)
\nn\\
&=\frac{s^\alpha}
{u^\alpha}\left(\frac{s^\alpha}{u^\alpha}\Nqa
(\fa1(t))-\frac{1}{u^\alpha}\fa1(0)\right)-\frac{1}{u^\alpha}
(\Dqa_t\fa1)(0)\nn\\
&=\left(\frac{s^\alpha}
{u^\alpha}\right)^2\Nqa(\fa1(t))-\frac{1}{u^\alpha}
\frac{s^\alpha}{u^\alpha}\fa1(0)-\frac{1}{u^\alpha}
(\Dqa_t\fa1)(0).\nn
\end{align}
Thus, we can state the following result.
\begin{theorem}\label{Transform-of-Derivative}
Suppose that function $\fa1(t)$ has polynomials or formal
power series expansion in $\alpha$-monomials $t^\alpha$. Then
for all integer $n>0$ it holds that
\begin{align*}
\Nqa((\Dqa_t)^n\fa1(t))=\left(\frac{s^\alpha}
{u^\alpha}\right)^n\Nqa(\fa1(t))-\frac{1}{u^\alpha}
\sum_{j=0}^{n-1}\left(\frac{s^\alpha}
{u^\alpha}\right)^{n-1-j}(\Dqa_t)^j\fa1(0).
\end{align*}
\end{theorem}
\begin{proof}
The proof is by induction on $n$. The Theorem statement
holds for $n=1$ as it shown in \eqref{TransformOfDerivative}.
Let us assume the formula holds for $k$, and let us prove it for
$k+1$. By using \eqref{TransformOfDerivative} we have
\begin{align*}
\Nqa((\Dqa_t)^{k+1}\fa1(t))=\frac{s^\alpha}{u^\alpha}\Nqa
((\Dqa_t)^k\fa1(t))-\frac{1}{u^\alpha}((\Dqa_t)^k\fa1)(0),
\end{align*}
and, by induction's assumption, we get
\begin{align}
\Nqa&((\Dqa_t)^{k+1}\fa1(t))
=-\frac{1}{u^\alpha}((\Dqa_t)^k\fa1)(0)\nn\\
&+\frac{s^\alpha}{u^\alpha}
\left(\left(\frac{s^\alpha}
{u^\alpha}\right)^k\Nqa(\fa1(t))-\frac{1}{u^\alpha}
\sum_{j=0}^{k-1}\left(\frac{s^\alpha}
{u^\alpha}\right)^{k-1-j}(\Dqa_t)^j\fa1(t)\right)
\nn\\
&=\left(\frac{s^\alpha}
{u^\alpha}\right)^{k+1}\Nqa(\fa1(t))-\frac{1}{u^\alpha}
\sum_{j=0}^{k-1}\left(\frac{s^\alpha}
{u^\alpha}\right)^{k-j}(\Dqa_t)^j\fa1(t)
-\frac{1}{u^\alpha}((\Dqa_t)^k\fa1)(0)\nn\\
&=\left(\frac{s^\alpha}
{u^\alpha}\right)^{k+1}\Nqa(\fa1(t))-\frac{1}{u^\alpha}
\sum_{j=0}^{k}\left(\frac{s^\alpha}
{u^\alpha}\right)^{k-j}(\Dqa_t)^j\fa1(t),\nn
\end{align}
which completes the proof.
\end{proof}
Let us consider now two examples of applying the conformable
fractional $q$-deformed natural transform for solving
differential equations.
\begin{example}
This example is extension of the Example~4.2.4 in
\cite{Debnath2007}. We have a differential equation
\begin{align*}
\left((\Dqa_t)^3+(\Dqa_t)^2-6\Dqa_t\right)f(t)=0,
\end{align*}
with the initial condition
\begin{align*}
f(0)=1,\quad \Dqa_tf(0)=0,\quad (\Dqa_t)^2f(0)=5.
\end{align*}
We apply our conformable fractional natural $q$-transform to
the differential equation, and, by using the
Theorem~\ref{Transform-of-Derivative}, obtain
\begin{align*}
\left(\frac{s^\alpha}{u^\alpha}\right)^3\bar{f}
&-\frac{1}{u^\alpha}\sum_{j=0}^2\left(\frac{s^\alpha}{u^\alpha
}\right)^{2-j}(\Dqa_t)^jf(0)+\left(\frac{s^\alpha}{u^\alpha}
\right)^2\bar{f}\nn\\&-\frac{1}{u^\alpha}\sum_{j=0}^1\left(
\frac{s^\alpha}{u^\alpha}\right)^{1-j}(\Dqa_t)^jf(0)
-6\cdot\frac{s^\alpha}{u^\alpha}\bar{f}+\frac{6}{u^\alpha}f(0)
=0,
\end{align*}
where $\bar{f}=\Nqa(f(t))$.
Let $w=\dfrac{s^\alpha}{u^\alpha}$. Then we get
\begin{align*}
w^3\bar{f}&-\frac{w^2}{u^\alpha}f(0)-\frac{w}{u^\alpha}
\Dqa_tf(0)-\frac{1}{u^\alpha}(\Dqa_t)^2f(0)+w^2\bar{f}\nn\\
&-\frac{w}{u^\alpha}f(0)-\frac{1}{u^\alpha}\Dqa_tf(0)
-6w\bar{f}+\frac{6}{u^\alpha}f(0)=0,
\end{align*}
and, by applying the initial conditions, we obtain
\begin{align*}
\bar{f}&=\frac{1}{u^\alpha}\frac{w^2+w-1}{w(w^2+w-6)}\nn\\
&=\frac{1}{u^\alpha}\left(\frac{1}{6w}+\frac{1}{3(w+3)}+
\frac{1}{2(w-2)}\right)\nn\\
&=\frac{1}{6}\frac{1}{s^\alpha}+\frac{1}{3}
\frac{1}{s^\alpha+3u^\alpha}+\frac{1}{2}\frac{1}
{s^\alpha-2u^\alpha}\nn\\
&=\frac{1}{6}\frac{1}{s^\alpha}+\frac{1}{3}
\frac{1}{s^\alpha-((-3)^\frac{1}{\alpha})^\alpha u^\alpha}
+\frac{1}{2}\frac{1}{s^\alpha-(2^\frac{1}{\alpha})^\alpha
u^\alpha}.
\end{align*}
Now, by using the results of the
Proposition~\ref{Transform-of_Exponentials}, we can find the
original function $f(t)$ as following
\begin{align*}
f(t)=\frac{1}{6}+\frac{1}{3}\eqa((-3)^\frac{1}{\alpha}t)
+\frac{1}{2}\eqa(2^\frac{1}{\alpha}t).
\end{align*}
One can easily see that this solution for $q=1$, $\alpha=1$
becomes $f(t)=\dfrac{1}{6}+\dfrac{1}{3}e^{-3t}+\dfrac{1}{2}
e^{2t}$, which coincides with the solution of
\cite{Debnath2007}.
\end{example}
\begin{example}
Let us consider now an extension of the differential equation
appearing in Example 4 of \cite{Zill2014}:
\begin{align*}
\Dqa_t f(t)+3f(t)=13\sqa(2^\frac{1}{\alpha}t),
\end{align*}
with the initial condition $f(0)=6$.
Again, let $\bar{f}=\Nqa(f(t))$. By applying the integral
transform to this differential equation, we obtain the
following equation
\begin{align*}
\frac{s^\alpha}{u^\alpha}\bar{f}-\frac{1}{u^\alpha}f(0)
+3\bar{f}=13\cdot\frac{2u^\alpha}{s^{2\alpha}+4u^{2\alpha}},
\end{align*}
which, by applying the initial condition, can be rewritten as
\begin{align*}
\frac{s^\alpha+3u^\alpha}{u^\alpha}\bar{f}=
\frac{26u^\alpha}{s^{2\alpha}+4u^{2\alpha}}+\frac{6}{u^\alpha}.
\end{align*}
Now we can express the transformation $\bar{f}$ as
\begin{align}
\bar{f}&=\frac{1}{s^\alpha+3u^\alpha}
\frac{5-u^{2\alpha}+6s^{2\alpha}}{s^{2\alpha}+4u^{2\alpha}}
\nn\\
&=\frac{A}{s^\alpha+3u^\alpha}+\frac{Bs^{\alpha} +Cu^{\alpha}}
{s^{2\alpha}+4u^{2\alpha}}.\label{fbar1}
\end{align}
The unknown constants $A$, $B$, $C$ can be found by comparing
two expressions for $\bar{f}$. One can easily check that
$A=8$, $B=-2$, and $C=6$. Therefore \eqref{fbar1}
can be rewritten as
\begin{align*}
\bar{f}=\frac{8}{s^\alpha+3u^\alpha}
-2\frac{s^\alpha}{s^{2\alpha}+4u^{2\alpha}}
+3\frac{2u^\alpha}{s^{2\alpha}+4u^{2\alpha}},
\end{align*}
where, by
Proposition~\ref{Transform-of_Exponentials} and
Proposition~\ref{Transform-of-Trigo}, we can obtain
the original function $f(t)$ as following
\begin{align*}
f(t)=8\eqa((-3)^\frac{1}{\alpha}t)-2\cqa(2^\frac{1}{\alpha}t)
+3\sqa(2^\frac{1}{\alpha}t).
\end{align*}
Note that for $q=1$, $\alpha=1$ we obtain the solution of
\cite{Zill2014}.
\end{example}

We have considered transforms of functions and
their derivatives. Let us consider now derivatives of the
transform. We would like to emphasize that a function $f(t)$
is, actually, polynomial or formal power series in
$t^\alpha$-monomials. Let us denote by $\Rqa(u,s)$ the
conformable fractional $q$-deformed natural transform
\eqref{Definition-Nqa}. With the notation
$\eqa^{-1}(t)=\frac{1}{\eqa(t)}$ we can state the following
lemma.
\begin{lemma} \label{Lemma1}
For all integer $n>0$,
\begin{align*}
(\Dqa_s)^n\eqa^{-1}(q^{-(n-1)}st)=(-1)^n t^{\alpha n}
q^{-\binom{n}{2}\alpha}\eqa^{-1}(qst).
\end{align*}
\end{lemma}
\begin{proof}
By applying the operator $\Dqa$ with respect to s consequently
$n-1$ times to the function $\eqa^{-1}(q^{-(n-1)}st)$ and using
\eqref{inverse-eqa-Derivative}, we obtain
\begin{align}
(\Dqa_s)^n\eqa^{-1}(q^{-(n-1)}st)
&=(-(q^{-(n-1)}t)^\alpha)\ldots(-(q^{-1}t)^\alpha)
\Dqa_s\eqa^{-1}(st).\nn 
\end{align}
Applying \eqref{inverse-eqa-Derivative} one more time, we
obtain
\begin{align*}
(\Dqa_s)^n\eqa^{-1}(q^{-(n-1)}st)
&=(-(q^{-(n-1)}t)^\alpha)\ldots(-(q^{-1}t)^\alpha)
(-t^\alpha)\eqa^{-1}(qst)\nn\\
&=(-1)^nt^{\alpha n}q^{-\binom{n}{2}\alpha}\eqa^{-1}(qst),
\end{align*}
wherefrom the Lemma's statement follows.
\end{proof}
\begin{proposition}
Suppose that function $\fa1(t)$ has polynomials or formal
power series expansion in $\alpha$-monomials $t^\alpha$. Then
for all integer $n>0$ it holds that
\begin{align*}
\Nqa(t^{\alpha n}\fa1(t))=(-1)^nq^{\binom{n}{2}\alpha}
u^{\alpha n}(\Dqa_s)^n\Rqa(u,q^{-n}s),
\end{align*}
where $\Rqa(u,s)=\Nqa(\fa1(t))$.
\end{proposition}
\begin{proof}
We have
\begin{align}
(\Dqa_s)^n\Rqa(u,q^{-n}s)&=(\Dqa_s)^n\int_0^\infty \fa1(ut)
\eqa^{-1}(q^{-(n-1)}st)\dqa t\nn\\
&=\int_0^\infty \fa1(ut)(\Dqa_s)^n
\eqa^{-1}(q^{-(n-1)}st)\dqa t.\nn
\end{align}
By using using Lemma~\ref{Lemma1}, we have
\begin{align}
(\Dqa_s)^n\Rqa(u,q^{-n}s)&=\int_0^\infty \fa1(ut)(-1)^n
t^{\alpha n}q^{-\binom{n}{2}\alpha}\eqa^{-1}(qst)\dqa t\nn\\
&=(-1)^n\frac{q^{-\binom{n}{2}\alpha}}{u^{\alpha n}}
\int_0^\infty (ut)^{\alpha n}\fa1(ut)\eqa^{-1}(q^{-(n-1)}st)
\dqa t,\nn\\
&=(-1)^n\frac{q^{-\binom{n}{2}\alpha}}{u^{\alpha n}}
\Nqa(t^{\alpha n}\fa1(t)).\nn
\end{align}
 The rearrangement of the last equation completes the proof.
\end{proof}

The Natural transform is a function of two variables, namely
$u$ and $s$. The previous Proposition establishes a connection
between the transform of product of $\fa1(t)$ with a positive
power of $\alpha$-monomials $t^\alpha$ and $q,\alpha$-deformed
derivative with respect to one of the variables, namely $s$,
of $q,\alpha$-transform. Let us consider now a derivative of
deformed transform with respect to its another variable $u$.
\begin{proposition}
Suppose that function $\fa1(t)$ has the following expansion:
$$\fa1(t)=\sum_{m=0}^{\infty}a_mt^{\alpha m}.$$ Then
for all integer $n>0$ it holds that
\begin{align*}
\Nqa(t^{\alpha n}\fa1(t))=\frac{u^{\alpha n}}{s^{\alpha n}}
(\Dqa_u)^nu^{\alpha n}\Rqa(u,s).
\end{align*}
\end{proposition}
\begin{proof}
If  $\fa1(t)=\sum_{m=0}^{\infty}a_mt^{\alpha m}$, then
\begin{align*}
\Nqa(t^{\alpha n}\fa1(t))&=\Nqa\left(\sum_{m=0}^{\infty}a_m
t^{\alpha (m+n)}\right).
\end{align*}
By using the linearity of the transform and applying
\eqref{transf-tan}, we get
\begin{align}
\Nqa(t^{\alpha n}\fa1(t))&=\sum_{m=0}^{\infty}
\frac{u^{\alpha(n+m)}}{s^{\alpha(n+m+1)}}[(n+m)\alpha]!a_m
\nn\\
&=\frac{u^{\alpha n}}{s^{\alpha n}}\sum_{m=0}^{\infty}
\frac{u^{\alpha m}}{s^{\alpha (m+1)}}[(n+m)\alpha]!a_m\nn\\
&=\frac{u^{\alpha n}}{s^{\alpha n}}\sum_{m=0}^{\infty}
(\Dqa_u)^n\frac{[m\alpha]!a_mu^{\alpha(n+m)}}{s^{\alpha(m+1)}}
\nn\\
&=\frac{u^{\alpha n}}{s^{\alpha n}}\sum_{m=0}^{\infty}
(\Dqa_u)^n\frac{u^{\alpha n}\cdot[m\alpha]!a_mu^{\alpha m}}
{s^{\alpha(m+1)}}\nn\\
&=\frac{u^{\alpha n}}{s^{\alpha n}}
(\Dqa_u)^nu^{\alpha n}\Nqa(\fa1(t))\nn.
\end{align}
The replacement $\Nqa(\fa1(t))$ by $\Rqa(u,s)$ in the last
equation completes the proof.
\end{proof}
These results are in complete agreement with those obtained
for non-deformed Sumudu and Natural transform investigated by
Belgacem and others \cite{Belgacem2006a,Belgacem2012}. Now we
will give another representation of the $q,\alpha$-deformed
natural transform of the product of $\fa1(t)$ with positive
degree of $\alpha$-monomial $t^\alpha$. This Proposition
extends the \cite[Theorem~4.2]{Belgacem2012}.
\begin{proposition}
Suppose that function $\fa1(t)$ has polynomial or formal
power series expansion in $\alpha$-monomials $t^\alpha$. Then
for all integer $n>0$ it holds that
\begin{align*}
\Nqa(t^{\alpha n}\fa1(t))=\frac{u^{\alpha n}}{s^{\alpha n}}
\sum_{k=0}^{n}b_{n,k}u^{\alpha k}(\Dqa_u)^k\Rqa(u,s),
\end{align*}
where the coefficients $b_{n,k}$ satisfy the recurrence
relationship
\begin{align*}
b_{n,k}=\left\{\begin{array}{ll}
[n\alpha]b_{n-1,0} &  k=0, \\
{[(n+k)\alpha]}b_{n-1,k}+q^{\alpha(n-1+k)}b_{n-1,k-1} & 0<k<n, \\
q^{\alpha(2n-1)}b_{n-1,n-1} & k=n.
\end{array}\right.
\end{align*}
with initial condition $b_{0,0}=1$.
\end{proposition}
\begin{proof}
We proceed the proof by induction on $n$.  For $n=0$ we have
$\Nqa(\fa1(t))=\Rqa(u,s)$, so that $b_{0,0}=1$. By previous
Proposition, for $n=1$ we have
\begin{align*}
\Nqa(t^\alpha\fa1(t))=\frac{u^\alpha}{s^\alpha}\Dqa_u
u^\alpha\Rqa(u,s),
\end{align*}
which, by applying deformed Leibniz rule
\eqref{qaLeibnizRule}, can be rewritten as
\begin{align*}
\Nqa(t^\alpha\fa1(t))=\frac{u^\alpha}{s^\alpha}
\left([\alpha]\Rqa(u,s)+q^\alpha u^\alpha\Dqa_u\Rqa(u,s)
\right).
\end{align*}
Thus $b_{1,0}=[\alpha]=[\alpha]b_{0,0}$,
$b_{1,1}=q^\alpha=q^\alpha b_{0,0}$, and the claim
holds. Assuming that the claim holds for $m\leq n$, we will
prove it for $m=n+1$.
\begin{align}
\Nqa(t^{\alpha(n+1)}\fa1(t))&=\Nqa(t^\alpha(t^{\alpha n}
\fa1(t)))\nn\\
&=\frac{u^\alpha}{s^\alpha}\Dqa_u u^\alpha
\frac{u^{\alpha n}}{s^{\alpha n}}\sum_{k=0}^{n}b_{n,k}
u^{\alpha k}(\Dqa_u)^k\Rqa(u,s)\nn\\
&=\frac{u^\alpha}{s^\alpha}\Dqa_u \sum_{k=0}^{n}b_{n,k}
\frac{u^{\alpha (1+n+k)}}{s^{\alpha n}}(\Dqa_u)^k\Rqa(u,s)
\nn\\
&=\frac{u^\alpha}{s^\alpha} \sum_{k=0}^{n}b_{n,k}
\frac{[(n+k+1)\alpha]u^{\alpha (n+k)}}{s^{\alpha n}}(\Dqa_u)^k
\Rqa(u,s)\nn\\
&+\frac{u^\alpha}{s^\alpha}\sum_{k=0}^{n}b_{n,k}
\frac{q^{\alpha (1+n+k)}u^{\alpha (1+n+k)}}{s^{\alpha n}}
(\Dqa_u)^{k+1}\Rqa(u,s)
\nn\\
&=\frac{u^{\alpha (n+1)}}{s^{\alpha (n+1)}} \sum_{k=0}^{n}b_{n,k}
[(n+k+1)\alpha]u^{\alpha k}(\Dqa_u)^k
\Rqa(u,s)\nn\\
&+\frac{u^{\alpha (n+1)}}{s^{\alpha (n+1)}}\sum_{k=1}^{n-1}
b_{n,k-1}q^{\alpha (n+k)}u^{\alpha k}
(\Dqa_u)^k\Rqa(u,s)
\nn\\
&=\frac{u^{\alpha (n+1)}}{s^{\alpha (n+1)}} \sum_{k=0}^{n+1}
b_{n+1,k}u^{\alpha k}(\Dqa_u)^k\Rqa(u,s),\nn
\end{align}
where
\begin{align}
\begin{array}{ll}
b_{n+1,0}&=b_{n,0}[(n+1)\alpha],\nn\\
b_{n+1,k}&=b_{n,k}[(n+k+1)\alpha]+b_{n,k-1}q^{\alpha(n+k)},
\quad 1\leq k\leq n,\nn\\
b_{n+1,n+1}&=b_{n,n}q^{\alpha(2n-1)},\nn
\end{array}
\end{align}
which completes the proof.
\end{proof}

We end this paper by the following conclusion. Our new
generalization of the Natural transform proposes also new
generalizations of other widely used integral transforms.
By applying the techniques described here, one can solve a
$k$-order linear $q$-differential equation with constant
coefficients. There is no need to find separately homogeneous
solution and a particular solution. In order to solve a
differential equation by applying the integral transform one
need to now the integral transform of the right-side function
of the differential equation $\sum_{0\leq j\leq k} a_{j}
(\Dqa_x)^{k-j} f(x)=b(x)$ and the initial conditions
$(\Dqa_x)^jy(0)=y_j$ for $j=0,\ldots,k-1$.

\vspace{0.5cm}
\textbf{Acknowledgement}. The research of the first author
was supported  by the Ministry of Science and Technology,
Israel.


\begin{thebibliography}{99}
\bibitem{Albayrak2013}
D. Albayrak, S.D. Purohit, and F. U\c{c}ar, On $q$-Sumudu
transforms of certain $q$-polynomials,
\emph{Filomat}, {\bf 27}(2) (2013), 411--427.
%
\bibitem{Albayrak2013a}
D. Albayrak, S.D. Purohit, and F. U\c{c}ar, On $q$-analogues of
Sumudu transform,
\emph{An. \c{S}tiin\c{t}. Univ. "Ovidius" Constan\c{t}a Ser. Mat.},
{\bf 21}(1) (2013), 239--260.
%
\bibitem{Belgacem2006}
F.B.M. Belgacem, Introducing and Analysing Deeper Sumudu
properties, \emph{Nonlinear Stud.}, {\bf 13}(1) (2006), 23--41.
%
\bibitem{Belgacem2006a}
F.B.M. Belgacem and A.A. Karaballi, Sumudu transform fundamental
properties investigations and applications,
\emph{ J. Appl. Math. Stoch. Anal.},
{\bf 2006}(4) (2006), Article ID 91083, 23 pp.
%
\bibitem{Belgacem2012}
F.B.M. Belgacem and R. Silambarasan, Advances in the Natural
transform, \emph{AIP Conf. Proc.}, {\bf 1493} (2012), 106--110.
%
\bibitem{Belgacem2012b}
F.B.M. Belgacem and R. Silambarasan, Theory of natural transform,
\emph{Journal | MESA},
{\bf 3}(1), 105--135, 2012.
%
\bibitem{Belgacem2012a}
F.B.M. Belgacem and R. Silambarasan, Maxwell's equations
solutions by means of the natural transform,
\emph{Journal | MESA}, {\bf 3}(3) (2012), 313-323.
%
\bibitem{Bohner2010}
M. Bohner and G.S. Guseinov, The $h$-Laplace and $q$-Laplace
transforms, \emph{J. Math. Anal. Appl.}, {\bf 365} (2010), 75--92.
%
\bibitem{Chung2014}
W.S. Chung, T. Kim, and H.I. Kwon, On the $q$-analogue of the
Laplace transform, \emph{Russ. J. Math. Phys.}, {\bf 21}(2) (2014), 156--168.
%
\bibitem{Chung2016}
W.S. Chung, On the $q$-deformed conformable fractional
calculus and the $q$-deformed generalized conformable
fractional calculus", preprint, 2016.
%
\bibitem{Debnath2007}
L. Debnath and D. Bhatta, \emph{Integral Transforms and Their
Applications}, 2nd. Ed., Chapman \& Hall/CRC, 2007.
%
\bibitem{Hahn1949}
W. Hahn, Beitrage Zur Theorie der Heineschen Re `ihen die 24
integral der. hypergeometrischen $q$-differenzeng Leichung, das
$q$-Analogon der Laplace-Transformation, \emph{Math.
Nachr}, {\bf 2} (1949), 340--379.
%
\bibitem{Kac2002}
V. Kac and P. Cheung, \emph{Quantum Calculus}, Springer, 2002.
%
\bibitem{Khan2008}
Z.H. Khan and W.A. Khan,   N-transform -- properties and
applications, \emph{NUST J. Eng. Sci.},
{\bf 1}(1) (2008), 127--133.
%
\bibitem{KK}
T. Kim and D.S. Kim, Degenerate Laplace transform and
degenerate gamma function, \emph{Russ. J. Math. Phys.},
{\bf 24}(2) (2017), 241--248.
%
\bibitem{Lenzi1999}
E.K. Lenzi, E.P. Borges, and R.S. Mendes, A $q$-generalization of
Laplace transforms,
\emph{J. Phys. A}, {\bf 32}(48) (1999), 8551--8562.
%
\bibitem{Plastino2013}
A. Plastino and M.C. Rocca,  The Tsallis-Laplace transform,
\emph{Phys. A}, {\bf 392} (2013), 5581--5591.
%
\bibitem{Purohit2007}
S.D. Purohit and S.L. Kalla, On $q$-Laplace transforms of the
$q$-Bessel functions,
\emph{Fract. Cal. Appl. Anal.}, {\bf 10}(2) (2007),
189--196.
%
\bibitem{Sim1}
Y. Simsek, Formulas for $p$-adic $q$-integrals including
falling-rising factorials, combinatorial sums and special
numbers, \emph{arXiv:1702.06999v1}, 2017.
%
\bibitem{Sim2}
Y. Simsek, Functional equations from generating functions:
a novel approach to deriving identities for the Bernstein
basis functions, \emph{Fixed Point Theory Appl.}, 2013
\textbf{2013}:80, 13 pp.
%
\bibitem{Ucar2012}
F. U\c{c}ar and D. Albayrak, On $q$-Laplace type integral
operators and their applications,
\emph{J. Difference Equ. Appl.},
{\bf 18}(6) (2012), 1001--1014.
%
\bibitem{Ucar2014}
F. U\c{c}ar, $q$-Sumudu transforms of $q$-analogues of Bessel
functions, \emph{The Scientific World J.}, 2014
Article ID 327019, 7 pp.
%
\bibitem{Watugala1993}
G.K. Watugala, Sumudu transform: a new integral transform to
solve differential equations and control engineering problems,
\emph{Internat. J. Math. Ed. Sci. Tech.}, {\bf 24}(1) (1993), 35--43.
%
\bibitem{Yadav2009}
R.K. Yadav, S.D. Purohit,and  P. Nirwan, On $q$-Laplace transforms
of a general class of $q$-polynomials and $q$-hypergeometric
functions, \emph{Math. Maced.}, {\bf 7} (2009), 81--88.
%
\bibitem{Zill2014}
D.G. Zill and W.S. Wright, \emph{Advanced Engineering
Mathematics}, 5th Ed., Jones \& Bartlett Learning, 2014.
\end{thebibliography}
\end{document}